\documentclass[11pt]{amsart}
\usepackage{amssymb, latexsym, mathrsfs, color, tikz}

\usepackage[colorlinks=true, pdfstartview=FitV,
 linkcolor=blue,citecolor=blue,urlcolor=blue]{hyperref}

\numberwithin{equation}{section}
\theoremstyle{plain}
\newtheorem{Thm}[equation]{Theorem}
\newtheorem{Prop}[equation]{Proposition}
\newtheorem{Cor}[equation]{Corollary}
\newtheorem{Lem}[equation]{Lemma}

\theoremstyle{definition}
\newtheorem{Def}[equation]{Definition}
\newtheorem{Exa}[equation]{Example}
\newtheorem{Rmk}[equation]{Remark}

\newcommand{\blue}{\textcolor{blue}}

\begin{document}

\title [Root Multiplicities of Rank $2$ hyperbolic Kac-Moody algebras]
{A Combinatorial Approach to Root Multiplicities \\ of Rank $2$ hyperbolic Kac-Moody algebras}
\author[S.-J. Kang]{Seok-Jin Kang$^{\star}$}
\thanks{$^{\star}$This work was supported by NRF Grant \# 2014-021261 and by NRF Grant \# 2010-0010753}
\address{Department of Mathematical Sciences and Research Institute of Mathematics,
Seoul National University, Seoul 151-747, South Korea}
\email{sjkang@snu.ac.kr}

\author[K.-H. Lee]{Kyu-Hwan Lee$^{\diamond}$}
\thanks{$^{\diamond}$This work was partially supported by a grant from the Simons Foundation (\#318706).}
\address{Department of
Mathematics, University of Connecticut, Storrs, CT 06269, U.S.A.}
\email{khlee@math.uconn.edu}

\author[K. Lee]{Kyungyong Lee$^{\dagger}$}
\thanks{$^{\dagger}$This work was partially supported by NSA grant H98230-14-1-0323}
\address{Department of
Mathematics, Wayne State University, Detroit, MI 48202, U.S.A., and Center for Mathematical Challenges, KIAS, Seoul 130-722, South Korea}
\email{klee@math.wayne.edu}

\begin{abstract}
In this paper we study root multiplicities of rank $2$ hyperbolic Kac-Moody algebras using the combinatorics of Dyck paths.
\end{abstract}

\maketitle

\section{Introduction}

This paper takes a new approach to the study of root multiplicities for hyperbolic Kac-Moody algebras. Even though the root multiplicities are fundamental data in understanding the structures of Kac-Moody algebras, we have not seen much progress in this topic for the last twenty years. The method taken in this paper is totally new, though depending on the previous developments, and opens different perspectives that can bring new results on root multiplicities and make advancements, for example, toward {\em Frenkel's conjecture}.
To begin with, let us first explain the backgrounds of the problem considered in this paper.

After introduced by Kac and Moody more than four decades ago, the Kac-Moody theory has become a standard generalization of the classical Lie theory. However, it makes one surprised to notice that little is known beyond the affine case. Even in the hyperbolic case, our knowledge is very limited in comparison with the affine case.

The first difficulty in the hyperbolic case and other indefinite cases stems from wild behaviors of root multiplicities. To be precise,
let $\mathfrak{g}$ be a  Kac-Moody algebra with Cartan subalgebra $\mathfrak{h}$. For a root $\alpha$, the root space $\mathfrak{g}_{\alpha}$ is given by
$$\mathfrak{g}_{\alpha}\ =\ \{x\in \mathfrak{g}\mid[h,x]=\alpha(h)x \text{ for all } h\in \mathfrak{h}\}.$$ Then we have the root space decomposition
$$\mathfrak{g}=\bigoplus_{\alpha\in\Delta^+}\mathfrak{g}_{\alpha}\ \oplus\ \mathfrak{h}\ \oplus\ \bigoplus_{\alpha\in\Delta^-}\mathfrak{g}_{\alpha},$$
which is a decomposition of $\mathfrak g$ into finite dimensional subspaces, where $\Delta^+$ (resp. $\Delta^-$) is the set of positive (resp. negative) roots.
The dimension of the root space $\mathfrak{g}_{\alpha}$ is called the {\it multiplicity} of $\alpha$.
Obviously, root multiplicities are fundamental data to understand the structure of a Kac-Moody algebra $\mathfrak g$. However, the status of our knowledge shows a dichotomy according to types of $\mathfrak g$.

Recall that the Weyl group $W$ of $\mathfrak g$ acts on the set $\Delta$ of all roots, preserving root multiplicities.
If $\alpha$ is a real root, $\alpha$ has an expression $\alpha=w\alpha_i$ for $w \in W$ where $\alpha_i$ is a simple root. It follows that $\dim(\mathfrak{g}_{\alpha})=1$. Since all roots in finite dimensional Lie algebras are real, all root spaces in finite dimensional Lie algebras are $1$ dimensional.
 Let $\mathfrak{g}$ be an  untwisted affine Kac-Moody algebra of rank $\ell+1$. Then the multiplicity of every imaginary root of  $\mathfrak{g}$ is $\ell$ (\cite[Corollary 7.4]{Kac}). There is a similar formula for twisted affine Kac-Moody algebras as well (\cite[Corollary 8.3]{Kac}).

For hyperbolic and more general indefinite Kac-Moody algebras, the situation is vastly different, due to the exponential growth of the imaginary root spaces.
Our knowledge of the dimensions of imaginary root spaces is far from being complete, though there are known formulas for root multiplicities.

The first formulas for root multiplicities of Kac--Moody algebras are a closed form formula by Berman and Moody (\cite{BM}) and a recursive formula by Peterson (\cite{P}).
Both formulas are based on the {\it denominator identity} for a Kac--Moody algebra $\mathfrak{g}$ and enable us to calculate the multiplicity of a given root (of a reasonable height). Computations of root multiplicities of hyperbolic Kac--Moody algebras began with the paper by Feingold and Frenkel \cite{FF}, where the hyperbolic Kac--Moody algebra $\mathfrak F$ of type $HA_1^{(1)}$ was considered. Using the same method,  Kac, Moody and Wakimoto \cite{KMW} calculated some root multiplicities for $HE_8^{(1)}(=E_{10})$.

These methods were further systematically developed and generalized by the first author \cite{Ka1, Ka2} for arbitrary Kac--Moody algebras and has been adopted in many works on roots multiplicities of indefinite Kac--Moody algebras. In his construction, the first author adopted homological techniques and Kostant's formula (\cite{GL}) to devise a method that works for higher levels. For example, he applied his method to compute roots multiplicities of the algebra $\mathfrak F$ of type $HA_1^{(1)}$ up to level $5$ (\cite{Ka3, Ka4}).

Despite all these results, we still do not have any unified, efficient approach to computing all root multiplicities. Essentially these methods give answers to root multiplicities one at a time, with no general formulas or effective bounds on multiplicities. In particular, these formulas are given by certain {\em alternating} sums of {\em rational} numbers and make it difficult to control overall behavior of root multiplicities. Therefore it is already quite hard to find effective upper or lower bounds for  root multiplicities for hyperbolic and other indefinite Kac-Moody algebras.

For hyperbolic Kac--Moody algebras, in the setting of the `no-ghost' theorem from String theory, I.  Frenkel \cite{F} proposed a bound on the root multiplicities of hyperbolic Kac--Moody algebras.

{\bf Frenkel's conjecture:} {\it Let ${\mathfrak {g}}$ be a symmetric hyperbolic Kac--Moody algebra associated to a hyperbolic lattice of dimension $d$ and equipped with invariant form $(\cdot\mid\cdot)$ such that $(\alpha_i | \alpha_i)=2$ for simple roots $\alpha_i$. Then we have:
$$\dim({\mathfrak {g}}_{\alpha})\ \leq\  p^{(d-2)} \left(1-\frac{(\alpha |\alpha)}{2}\right),$$
where the function $p^{(\ell)}(n)$  is the multi-partition function with $\ell$ colors. }

Frenkel's conjecture is known to be true for any symmetric Kac-Moody
algebra associated to a hyperbolic lattice of dimension $26$ \cite{F}, though Kac, Moody and Wakimoto
\cite{KMW} showed that the conjecture fails for $E_{10}$. The conjecture
is still open for the rank $3$ hyperbolic Kac--Moody algebra
$\mathfrak F$ and proposes arguably the most tantalizing question
about root multiplicities.

\medskip
{\bf Open Problem:}
{\it Prove Frenkel's conjecture for the rank $3$ hyperbolic Kac-Moody algebra $\mathfrak F$.}
\medskip

As mentioned earlier, Feingold and Frenkel \cite{FF} and the first author \cite{Ka3, Ka4} studied root multiplicities of $\mathfrak F$. 
There is another approach to root multiplicities of $\mathfrak F$ and other hyperbolic Kac-Moody algebras,  taken by Niemann \cite{Nie}, which follows  Borcherds' idea in construction of the fake Monster Lie algebra \cite{Bor}. This approach was further pursued by Kim and the second author \cite{KL}. A recent survey on root multiplicities can be found in \cite{CFL}.

In this paper, we adopt quite a different methodology  and investigate root multiplicities of  rank two symmetric hyperbolic Kac-Moody algebras $\mathcal H(a)$ $(a \ge 3)$ through combinatorial objects. More precisely, we use lattice paths, known as {\em Dyck paths}, to describe root multiplicities.

Suppose that  $\alpha= r\alpha_1 + s \alpha_2$ is an imaginary root of $\mathcal H(a)$ with $r$ and $s$ relatively prime, for simplicity. Then our first main theorem (Theorem \ref{thm-ff}) shows that

\begin{Thm}\label{coprime-thm}We have
\[ \mathrm{mult}\, (\alpha) = \sum_{\substack{ D: \, \text{\rm Dyck path} \\ \mathrm{wt} (D)=\alpha} } c(D). \]
\end{Thm}
Here $c(D)$ has values $1$, $0$ or $-1$ and is immediately determined by the shape of the Dyck path $D$. The result for general $\alpha$ involves considering cyclic equivalence of paths and a minor correction term coming from paths with weight $\alpha/2$. An important feature is that  this formula only contains integers and has clear combinatorial interpretation, and makes it possible to prove properties of root multiplicities through combinatorial manipulations of Dyck paths. 
For example, in the symmetric rank two case, we can prove an analogue of Frenkel's conjecture through combinatorics of Dyck paths.

\begin{Prop}
Let $\mathfrak {g}=\mathcal H(a)$. Then we have:
$$\mathrm{mult}\, (\alpha) \leq\  p_t \left(1-\frac{(\alpha |\alpha)}{2}\right),$$
where $\alpha=r \alpha_1 + s\alpha_2$, $t=\max (r,s)$ and  $p_t(n)$  is the number of partitions of $n$ with at most $t$ parts.
\end{Prop}

Even though this upper bound is in the form of Frenkel's conjecture, it is actually crude. More interestingly, Theorem~\ref{coprime-thm} gives a natural upper bound by only counting paths with $c(D)=1$. This upper bound can be significantly improved by considering cancellation with paths having $c(D)=-1$. Namely, we consider a function $\Phi$ from $\{D :\,c( D)=-1 \}$ to $\{D :\,c( D)=1\}$. 
Suppose that  $\alpha= r\alpha_1 + s \alpha_2$ is an imaginary root of $\mathcal H(a)$ with $r$ and $s$ relatively prime, for simplicity. Then we obtain

\begin{Thm} \[   \mathrm{mult}\, (\alpha) \le \# \{ D :   \, \text{\rm Dyck path}, \,  \mathrm{wt} ( D)=\alpha, \,  c( D)=1, \, D \text{\rm \  is not an image under }\Phi \} .\]
\end{Thm}

This upper bound is quite sharp and gives exact root multiplicities for roots up to height 16 with a suitable choice of $\Phi$. In
Section~\ref{section-sharper}, the function $\Phi$ will be carefully constructed. 
The resulting upper bound is satisfactorily accurate and enlightens combinatorics of Dyck paths related to root multiplicities.

Our approach clearly  extends to higher rank  Kac-Moody algebras by replacing Dyck paths with certain lattice paths. In a subsequent paper, we will consider higher rank cases; in particular, we will study the Feingold-Frenkel rank $3$ algebra $\mathfrak F$. We hope that our approach may bring  significant advancements toward Frenkel's conjecture for the algebra $\mathfrak F$.


\vskip 1cm

\section{Rank Two Symmetric Hyperbolic Kac-Moody Algebras} \label{hyperbolic}

In this section, we fix our notations for rank $2$ hyperbolic
Kac-Moody algebras. A general theory of Kac-Moody algebras can be
found in \cite{Kac}, and the root systems of rank two hyperbolic
Kac-Moody algebras were studied by Lepowsky and Moody \cite{LM} and
Feingold \cite{Fein}. Root multiplicities of these algebras were
investigated by Kang and Melville \cite{KaMe}.

Let $A=\begin{pmatrix} 2 & -a \\ -a & 2 \end{pmatrix}$ be a
generalized Cartan matrix with $a \ge 3$, and  $\mathcal H(a)$ be
the hyperbolic Kac-Moody algebra associated with the matrix $A$. In
this section, we write $\mathfrak g = \mathcal H(a)$ if there is no
need to specify $a$. Let $\{ h_1, h_2 \}$ be the set of simple
coroots in the Cartan subalgebra $\mathfrak h = \mathbb C h_1 \oplus
\mathbb C h_2 \subset \mathfrak g$. Let $\{ \alpha_1 , \alpha_2 \}
\subset \mathfrak h^*$ be the set of simple roots, and $Q=\mathbb Z
\alpha_1 \oplus \mathbb Z \alpha_2$ be the root lattice.  The set of
roots of $\mathfrak g$ will be denoted by $\Delta$, and the set of
positive (resp. negative) roots by $\Delta^+$ (resp. by $\Delta^-$),
and the set of real (resp. imaginary) roots by
$\Delta_{\mathrm{re}}$ (resp. by $\Delta_{\mathrm{im}}$). We will
use the notation $\Delta^+_{\mathrm{re}}$ to denote the set of
positive real roots. Similarly, we use $\Delta^+_{\mathrm{im}}$,
$\Delta^-_{\mathrm{re}}$ and $\Delta^-_{\mathrm{im}}$. The Lie
algebra $\mathfrak g$ has the root space decomposition $\mathfrak g
= \mathfrak h \oplus \bigoplus_{\alpha \in \Delta} \mathfrak g_{\alpha}$ and we define the
{\em multiplicity} of $\alpha$ by $\mathrm{mult} \, \alpha := \dim
\mathfrak g_\alpha$.

We define a symmetric bilinear form on $\mathfrak h^*$ by $(\alpha_i | \alpha_j)=a_{ij}$, where $a_{ij}$ is the $(i,j)$-entry of the Cartan matrix $A$. The simple reflection corresponding to $\alpha_i$ in the root system of $\mathfrak g$ is denote by $r_i$ ($i=1, 2$), and the Weyl group $W$ is given by $W=\{ (r_1r_2)^i, r_2(r_1r_2)^i \, | \, i \in \mathbb Z \}$. Define  a sequence $\{ B_n \}$ by
\[ B_0=0, \quad B_1=1, \quad B_{n+2} = a B_{n+1} - B_n \quad \text{ for } n \ge 0. \] It can be shown that \[ B_n =  \frac {1 -\gamma^{2n}} { \gamma^{n-1}(1-\gamma^2)} \qquad (n \ge 0) ,\] where $\gamma = \frac {a +\sqrt{a^2-4}}{2}$. We will write $(A,B)=A \alpha_1 + B \alpha_2$. Then the set of real roots are given by
\[ \Delta^+_{\mathrm{re}} = \left \{ (B_n,B_{n+1}), \ ( B_{n+1},B_{n}) \ | \ n \ge 0   \right \} . \]
See \cite{KaMe} for details.
To describe the set of imaginary roots, we first define the set
\[ \Omega_k = \left \{ (m,n) \in \mathbb Z_{\ge 0} \times \mathbb Z_{\ge 0} : \sqrt{\frac {4 k}{a^2 -4} } \le m \le \sqrt {\frac k {a-2} },\ n = \frac { am-\sqrt{(a^2-4)m^2 -4k}} 2 \right \} \]  for $k \ge 1$.
\begin{Prop} \cite{KaMe} For $a \ge 3$, the set of positive imaginary roots $\alpha$ of $\mathcal H (a)$ with $(\alpha |\alpha)=-2k$ is
\begin{equation*} \label{imroot} \Delta^+_{\mathrm{im}, k} = \left \{ \begin{array}{l} (m,n), \ (mB_{j+1}-nB_j, mB_{j+2}-nB_{j+1} ), \\ (mB_{j+2}-nB_{j+1}, mB_{j+1}-nB_j ) \end{array} \ \Big | \ (m,n) \in \Omega_k  \text{ or } (n,m) \in \Omega_k  , \  j \ge 0  \right \}. \end{equation*}
\end{Prop}

The denominator identity is given by
\[ \prod_{\alpha \in \Delta^+} (1 -e^{- \alpha})^{\mathrm{mult} \, \alpha} = \sum_{ w \in W} (-1)^{\ell(w)} e^{w \rho - \rho} ,\] where $\ell(w)$ is the length of $w$ and $\rho= (\alpha_1 +\alpha_2)/(2-a)$.

\vskip 1 cm

\section{Contribution multiplicity}  \label{contribution}

In this section, we fix a hyperbolic Kac-Moody algebra $\mathcal H(a)$, $a \ge 3$. First, we recall Kang and Melville's result \cite{KaMe} on root multiplicities of $\mathcal H(a)$. For $r, s \in \mathbb Z_{\geq 0}$, write $\alpha= r \alpha_1 + s \alpha_2$.
As in \cite{KaMe}, we define a sequence $\{A_n\}_{n\geq 0}$
    as follows:
    $$
A_0=0, \quad A_1=1,$$
$$A_{n+2}=aA_{n+1}-A_n +1 \text{ for }n\geq 0.$$

 Let
$$
\mathscr{C}=\{\mathbf{c}=(c_0^0, c_0^1,c_1^0, c_1^1,...)\ | \ c_i^j \text{ are non-negative integers}, j\in\{0,1\},i\geq 0\},
$$
and let \begin{equation} \label{eqn-c} \mathscr C(\alpha)=\{ \mathbf c \in \mathscr C \, | \,  \sum_{i \ge 0} (c^0_i A_{i+1} + c^1_i A_i)=r, \  \sum_{i \ge 0} (c^0_i A_{i} + c^1_i A_{i+1})=s \}.\end{equation} We write $\tau | \alpha$ if $\alpha= d \tau$ for some $d \in \mathbb Z_{>0}$, and set $\alpha/\tau =d$.

\begin{Prop}\cite[Proposition 2.1, 2.2]{KaMe} \label{KM}
We have
\begin{equation} \label{KM-for} \mathrm{mult} \, (\alpha) = \sum_{\tau | \alpha} \mu \left ( \frac{\alpha}{\tau} \right ) \frac{\tau}{\alpha} \sum_{\mathbf c \in \mathscr C(\tau)} (-1)^{\sum_{i : \mathrm{odd}} (c^0_i+c^1_i)} \frac{( \sum_{i \ge 0} (c^0_i+c^1_i) -1)!} {\prod_{i \ge 0} c^0_i ! c^1_i !} .\end{equation}
\end{Prop}

For $r, s \in \mathbb Z_{\geq 0}$, define a {\em Dyck path of size $r \times s$} to be a lattice path from $(0,0)$ to $(r,s)$ that never goes above the main diagonal joining $(0,0)$ and $(r,s)$. We identify a Dyck path with a word in alphabet $\{ 1, 2 \}$, where $1$ represents a horizontal move and $2$ a vertical move.  Then a  Dyck path has $12$-corners and $21$-corners. We consider the end points $(0,0)$ and $(r,s)$ as $21$-corners.
We define the {\em weight} of a Dyck path $D$ of size $r \times s$ to be
$\mathrm{wt} (D) := r \alpha_1 +s \alpha_2 \in Q$.

We say that  two Dyck paths $D_1$ and $D_2$ are {\em equivalent} if $D_1$  is a cyclic permutation of $D_2$ (as words in the alphabet $\{ 1,2 \}$). Then we obtain {\em equivalent classes} of Dyck paths. When no confusion arises, we will frequently identify an equivalent class $\mathcal D$ with any representative $D\in\mathcal D$.        For an equivalence class $\mathcal D$, the {\em weight} $\mathrm{wt}(\mathcal D)$ is well-defined. The concatenation of  Dyck paths $D_1, D_2, \dots, D_r$ will be denoted by $D_1D_2 \cdots D_r$. For a positive integer $d$ and a Dyck path $D$, the concatenation $D^d$ is defined in an obvious way. We distinguish a concatenation from its resulting path. The resulting path of a concatenation $D_1D_2 \cdots D_r$ will be denoted by $\pi(D_1D_2 \cdots D_r)$. A Dyck path $D$ is said to be {\em essential} if $D \neq \pi(D_0^d)$ for any subpath $D_0$ and $d \ge 2$. Likewise, an equivalence class $\mathcal D$ is said to be {\em essential} if any element $D$ of $\mathcal D$ is essential.

\begin{Def}
For any positive integers $u,v$, denote by $L_{u\times v}$ the Dyck path of size $u\times v$, which consists of $u$ horizontal edges followed by $v$ vertical edges, and call it an {\em elementary} path.
We say that the elementary Dyck path $L_{u\times v}$ is of
$$\left\{\begin{array}{ll}
\text{ type }(-1), & \text{ if }A_n\leq \min(u,v), \  \max( u,v) <A_{n+1},\text{ and }n:\text{even}>0;\\
\text{ type } (1), & \text{ if }A_n\leq \min(u,v), \ \max( u,v) <A_{n+1},\text{  and }n:\text{odd};\\
\text{ type } (0), & \text{ if }A_n\leq \min(u,v) <A_{n+1}\leq\max(u,v),\text{  and }n>0.
\end{array}\right.$$
\end{Def}

For a given Dyck path $D$, define  $\mathscr{S}(D)$ to be the set of  all concatenations of copies of $L_{A_{i+1}\times A_{i}}$  and copies of $L_{A_{i}\times A_{i+1}}$ in some order that realize $D$.
For a concatenation $\mathbf s$ in $\mathscr{S}(D)$, the number of copies of $L_{A_{i+1}\times A_{i}}$ is denoted by $c_i^0(\mathbf s)$ and the number of copies of $L_{A_{i}\times A_{i+1}}$ by $c_i^1(\mathbf s)$. We define \[ \mathrm{seq}(\mathbf s) = (c_i^0(\mathbf s), c_i^1(\mathbf s))_{i\ge 0} \in \mathscr C \quad \text{ and } \quad \mathrm{sgn}(\mathbf s) =(-1)^{\sum_{i:\text{odd}}
(c_i^0(\mathbf s)+c_i^1(\mathbf s))}. \]
 If $\mathcal D$ is an equivalence class, we observe that $\mathscr S (D_1)$ is in one-to-one correspondence with  $\mathscr S(D_2)$ through cyclic permutation for $D_1, D_2 \in \mathcal D$. For an equivalence class $\mathcal D$, we define the set $\mathscr S(\mathcal D)$ to be equal to $\mathscr S(D)$ for a fixed Dyck path  $D \in \mathcal D$.
Now the {\em contribution multiplicity} $c(\mathcal D)$ of $\mathcal D$ is defined  by
$$
c(\mathcal D)=\sum_{\mathbf s \in \mathscr{S}(\mathcal D)}\mathrm{sgn}(\mathbf s).
$$

For a Dyck path $D$, a subpath $D_0$ of $D$  is called \emph{framed} if the starting point and the ending point of $D_0$ are both $21$-corners.

\begin{Lem} \label{lem-c}
For any Dyck path $D$, we have \[c(D)= \begin{cases} 0, \quad \text{if $D$ contains a framed subpath of type $(0)$}; \\ (-1)^{\# \text{ of framed subpaths of $D$ of type $(-1)$}}, \quad \text{otherwise}.\end{cases} \]
\end{Lem}

\begin{proof}
Assume that $D=\pi(D_1D_2)$. Then we have
\begin{align*}
c(D)=& \sum_{\mathbf s  \in \mathscr S(D)} \mathrm{sgn}(\mathbf s) = \sum_{(\mathbf s_1, \mathbf s_2) \in \mathscr S(D_1) \times \mathscr S(D_2)} \mathrm{sgn}(\mathbf s_1) \ \mathrm{sgn}(\mathbf s_2) \\ =& \sum_{\mathbf s_1 \in \mathscr S(D_1)} \mathrm{sgn}(\mathbf s_1)  \sum_{\mathbf s_2 \in \mathscr S(D_2)} \mathrm{sgn}(\mathbf s_2) = c(D_1) c(D_2).
\end{align*}
Thus it is enough to consider the case when $D$ is an elementary path. In this case, we need to prove that $c(D)$ is equal to its type.  We will use induction. Clearly, $c(L_{1 \times 0})=c(L_{0 \times 1})=1$, and the assertion of the lemma is true. Suppose that the assertion is true for $L_{u \times v}$. We will prove the case $L_{(u+1) \times v}$. The other case $L_{u \times (v+1)}$ is obtained from the symmetry.

Write $L=L_{u \times v}$ and $L_1=L_{(u+1) \times v}$ to ease the notations. Assume that $c(L)$ is equal to its type. If $L$ and $L_1$ is of the same type, then we get all the elements of $\mathscr S(L_1)$ from those of $\mathscr S(L)$ by adding $1$ to $c_0^0(\mathbf s)$, $\mathbf s \in \mathscr S(L)$, and $c(L_1)=c(L)$.

If $L$ is of type $(1)$ and $L_1$ is of type $(0)$, then the path $L_1$ newly contains $L_{A_{n+1} \times A_n}$ as a subpath for $n$ odd, where $A_{n+1}=u+1$. Consequently, $c(L_1)=c(L)-1=0$ by induction. If $L$ is of type $(-1)$ and $L_1$ is of type $(0)$, then the path $L_1$ newly contains $L_{A_{n+1} \times A_n}$ as a subpath for $n$ even, where $A_{n+1}=u+1$. Thus, again, we have $c(L_1)=c(L)+1=0$.

Similarly, if $L$ is of type $(0)$ and $L_1$ is of type $(-1)$ (respectively, if $L$ is of type $(0)$ and $L_1$ is of type $(1)$), then $L_1$ newly contains $L_{A_{n+1} \times A_n}$ as a subpath for $n$ odd (respectively, for $n$ even), where $A_{n+1}=u+1$. Thus we have $c(L_1)=c(L)-1=-1$ (respectively, $c(L_1)=c(L)+1=1$). Now, by induction, we are done.
\end{proof}

\begin{Rmk}
The above lemma enables us to compute $c(D)$  efficiently and combinatorially. In particular, $c(D)=1$ if $D$ contains no framed subpaths of type $(0)$ and an even number of  framed subpaths of type $(-1)$.
\end{Rmk}

The following theorem is a combinatorial realization of Kang and Melville's formula \eqref{KM-for}, which says that the root multiplicity of $\alpha$ is equal to the sum of contribution multiplicities $c(\mathcal D)$ of essential equivalence classes $\mathcal D$ of weight $\alpha$ plus some correction term.

\begin{Thm} \label{thm-ff} For $\alpha \in \Delta^+$,  we have
\[ \mathrm{mult}\, (\alpha) = \sum_{\substack{\mathcal D: \, \text{\rm essential} \\ \mathrm{wt} (\mathcal D)=\alpha} } c(\mathcal D) +\sum_{\substack{\mathcal D: \, \text{\rm essential} \\ \mathrm{wt} (\mathcal D)=\alpha/2} } \left \lfloor \tfrac { 1-c(\mathcal D)} 2 \right \rfloor . \]
\end{Thm}
By Lemma \ref{lem-c}, we see that the second sum (i.e., the correction term) is nothing but the number  of essential $\mathcal D$ such that  $\mathrm{wt}(\mathcal D) = \alpha/2$ and $c (\mathcal D)=-1$.

\begin{proof}
Write $\alpha = r \alpha_1 + s \alpha_2$.
Before we deal with the general case,  we first consider a simpler case and assume that $r$ and $s$ are relatively prime. Then the correction term is $0$, and each equivalence class of weight $\alpha$ has only one essential Dyck path. Recall that we defined $\mathscr C(\alpha)$ in \eqref{eqn-c}.
We claim that, for each $\mathbf c=(c_i^0, c_i^1)_{i\ge 0} \in \mathscr C(\alpha)$,  the number of concatenations $\mathbf s$ such that $\mathrm{seq}(\mathbf s) =\mathbf c$ is $\displaystyle{\frac{ (\sum_{i \ge 0} (c^0_i+c^1_i) -1)!} {\prod_{i \ge 0} c^0_i ! c^1_i !} }$. Indeed, let $\mathbf p$ be a concatenation of the $c_i^0$  copies of $L_{A_{i+1}\times A_{i}}$  and $c_i^1$  copies of $L_{A_{i}\times A_{i+1}}$ in some order, and consider the concatenation $\mathbf p^N$ for $N$ sufficiently large. Then we can find a unique  line with slope $s/r$ which intersects the path $\pi(\mathbf p^N)$ so that the path never goes above the line. Since $r$ and $s$ are relatively prime, two consecutive intersection points uniquely determine a concatenation which is a cyclic permutation of $\mathbf p$, and the number of cyclic permutations is  $\sum_{i \ge 0} (c^0_i+c^1_i)$. Now the claim  follows.

From Proposition \ref{KM} and the claim above, we obtain
\begin{align*} \sum_{\substack{\mathcal D: \, \text{\rm essential} \\ \mathrm{wt} (\mathcal D)=\alpha} } c(\mathcal D) &= \sum_{\mathcal D:\, \mathrm{wt} (\mathcal D)=\alpha}\  \sum_{\mathbf s \in \mathscr S (\mathcal D)} (-1)^{\sum_{i:\text{odd}}(c_i^0(\mathbf s)+c_i^1(\mathbf s))} \\ & = \sum_{\mathbf c \in \mathscr C(\alpha) }  (-1)^{\sum_{i:\text{odd}} (c_i^0+c_i^1)} \frac{ (\sum_{i \ge 0} (c^0_i+c^1_i) -1)!} {\prod_{i \ge 0} c^0_i ! c^1_i !} = \mathrm{mult} \, (\alpha).  \end{align*}

Now we consider arbitrary $r, s \in \mathbb Z_{\ge 0}$. We will show
\begin{equation} \label{eqn-th1}  \sum_{\mathbf c \in \mathscr C (\alpha)} (-1)^{\sum_{i:\text{odd}} (c_i^0+c_i^1)} \frac{ (\sum_{i \ge 0} (c^0_i+c^1_i) -1)!} {\prod_{i \ge 0} c^0_i ! c^1_i !} = \sum_{\tau | \alpha} \frac \tau \alpha \left [ \sum_{\substack{\mathcal D: \, \text{\rm essential} \\ \mathrm{wt} (\mathcal D)=\tau} } c(\mathcal D) +\sum_{\substack{\mathcal D: \, \text{\rm essential} \\ \mathrm{wt} (\mathcal D)=\tau/2} } \left \lfloor \tfrac { 1-c(\mathcal D)} 2 \right \rfloor \right ] .\end{equation}  Let $\mathbf c \in \mathscr C(\alpha)$. As before, assume that  $\mathbf p$ is a concatenation of the $c_i^0$  copies of $L_{A_{i+1}\times A_{i}}$  and $c_i^1$  copies of $L_{A_{i}\times A_{i+1}}$ in some order, and consider the concatenation $\mathbf p^N$ for $N$ sufficiently large. Then we can find a unique  line with slope $s/r$ which intersects the path $\pi(\mathbf p^N)$ so that the path never goes above the line. Then we obtain an equivalence class of concatenations of size $r \times s$.  We choose a concatenation from the equivalence class and denote it again by $\mathbf p$.

If $\mathbf p=\mathbf p_0^d$ for some concatenation $\mathbf p_0$ of weight $\tau$ such that $\alpha/\tau=d$ and $d$ is maximal, then the number of cyclic permutations of $\mathbf p$ is  $\sum_{i \ge 0} (c^0_i+c^1_i)/d$. Define the {\em contribution} of  the equivalence class of $\mathbf p$ to be  $(-1)^{\sum_{i:\text{odd}} (c_i^0+c_i^1)}/d$. Then the total sum of contributions of equivalence classes of concatenations $\mathbf p$ such that $\mathrm{seq}(\mathbf p)=\mathbf c$ is given by
 $(-1)^{\sum_{i:\text{odd}} (c_i^0+c_i^1)} \frac{ (\sum_{i \ge 0} (c^0_i+c^1_i) -1)!} {\prod_{i \ge 0} c^0_i ! c^1_i !}$. One can see this by observing that $\frac{ (\sum_{i \ge 0} c^0_i+c^1_i )!} {\prod_{i \ge 0} c^0_i ! c^1_i !}$ counts the number of concatenations and that $\frac{ (\sum_{i \ge 0} (c^0_i+c^1_i)- 1)!} {\prod_{i \ge 0} c^0_i ! c^1_i !}$ is the {\em weighted} number of cyclic equivalence classes of concatenations when we assign a weight $1/d$ to an equivalence class of $\sum_{i \ge 0} (c^0_i+c^1_i)/d$ members.

We group the equivalence classes of concatenations $\mathbf p$ according to the resulting equivalence classes $\mathcal D$ of Dyck paths so that $\pi(\mathbf p)\in \mathcal D$, and define $\mathcal T_{\mathcal D}$ to be  the total sum  of contributions of the equivalence classes of $\mathbf p$ such that $\pi(\mathbf p)\in \mathcal D$. Then we have
\begin{equation} \label{eqn-ee}
 \sum_{\mathbf c \in \mathscr C (\alpha)} (-1)^{\sum_{i:\text{odd}} (c_i^0+c_i^1)} \frac{ (\sum_{i \ge 0} (c^0_i+c^1_i) -1)!} {\prod_{i \ge 0} c^0_i ! c^1_i !} = \sum_{\mathcal D : \mathrm{wt} (\mathcal D) =\alpha} \mathcal T_{\mathcal D} . \end{equation}

We consider an equivalence class $\mathcal D$ of Dyck paths of weight $\alpha$ and choose a representative $D$. Let $D_0$ be an essential subpath of $D$  such that $D=\pi(D_0^{d})$, and let $\mathscr S(D_0)=\{ \mathbf s_1, \dots , \mathbf s_k \}$.  If $\mathbf p$ is a concatenation such that $\pi(\mathbf p)=D$ then $\mathbf p$ is equal to a concatenation of $d$ choices  of $\mathbf s_i$ from $\mathscr S(D_0)$ with repetition allowed. Thus the total sum $\mathcal T_{\mathcal D}$ of contributions  is equal to
\[\mathcal T_{\mathcal D}= \tfrac 1 d (\mathrm{sgn}(\mathbf s_1)+\cdots + \mathrm{sgn}(\mathbf s_k) )^d= \tfrac 1 d c(D_0)^d.\]
By Lemma \ref{lem-c}, we know that $c(D_0)=-1, 0$ or $1$. Unless $c(D_0)=-1$ and $d$ is even, we have $\mathcal T_{\mathcal D}= \frac 1 d c(D_0)$.
If $c(D_0)=-1$ and $d$ is even, then  we have $\mathcal T_{\mathcal D} = \frac 1 d= \frac 1 d c(D_0) + \frac 2 d$.

Now we obtain
\begin{align*}  \sum_{\mathcal D : \mathrm{wt} (\mathcal D) =\alpha} \mathcal T_D =&  \sum_{\tau | \alpha} \frac \tau \alpha  \sum_{\substack{\mathcal D_0: \, \text{\rm essential} \\ \mathrm{wt} (\mathcal D_0)=\tau} } c(\mathcal D_0) +\sum_{2\tau | \alpha} \frac{2 \tau}{\alpha} \sum_{\substack{\mathcal D_0: \, \text{\rm essential} \\ \mathrm{wt} (\mathcal D_0)=\tau} } \delta_{c(\mathcal D_0)+1,0} \\ =&\sum_{\tau | \alpha} \frac \tau \alpha \left [ \sum_{\substack{\mathcal D: \, \text{\rm essential} \\ \mathrm{wt} (\mathcal D)=\tau} } c(\mathcal D) +\sum_{\substack{\mathcal D: \, \text{\rm essential} \\ \mathrm{wt} (\mathcal D)=\tau/2} } \left \lfloor \tfrac { 1-c(\mathcal D)} 2 \right \rfloor \right ] ,
\end{align*}
where $\delta$ is the Kronecker delta. Combined with \eqref{eqn-ee}, this establishes the desired identity \eqref{eqn-th1}.

Finally, let $\beta=r_0 \alpha_1 + s_0 \alpha_2$ be such that  $r_0$ and $s_0$ are relatively prime and $\beta | \alpha$. Multiplying both sides of \eqref{eqn-th1} by $\alpha/\beta$, we obtain
\begin{equation} \frac \alpha \beta  \sum_{\mathbf c \in \mathscr C (\alpha)} (-1)^{\sum_{i:\text{odd}} (c_i^0+c_i^1)} \frac{ (\sum_{i \ge 0} (c^0_i+c^1_i) -1)!} {\prod_{i \ge 0} c^0_i ! c^1_i !} = \sum_{\tau | \alpha} \frac \tau \beta \left [ \sum_{\substack{\mathcal D: \, \text{\rm essential} \\ \mathrm{wt} (\mathcal D)=\tau} } c(\mathcal D) +\sum_{\substack{\mathcal D: \, \text{\rm essential} \\ \mathrm{wt} (\mathcal D)=\tau/2} } \left \lfloor \tfrac { 1-c(\mathcal D)} 2 \right \rfloor \right ] .\end{equation}
It follows from the M\"obius inversion and Proposition \ref{KM} that
\[ \mathrm{mult}\, (\alpha) = \sum_{\substack{\mathcal D: \, \text{\rm essential} \\ \mathrm{wt} (\mathcal D)=\alpha} } c(\mathcal D) +\sum_{\substack{\mathcal D: \, \text{\rm essential} \\ \mathrm{wt} (\mathcal D)=\alpha/2} } \left \lfloor \tfrac { 1-c(\mathcal D)} 2 \right \rfloor . \] This completes the proof.
\end{proof}

As a corollary,  we can prove an analogue of Frenkel's conjecture.

\begin{Cor} \label{Frenkel-bounds}
We have
$$\mathrm{mult}\, (\alpha)\ \leq\  p_t \left(1-\frac{(\alpha |\alpha)}{2}\right),$$
where $\alpha=r \alpha_1 + s\alpha_2$, $t=\max (r,s)$ and  $p_t(n)$  is the number of partitions of $n$ with at most $t$ parts.
\end{Cor}
\begin{proof}
Let $\alpha=r\alpha_1+s\alpha_2$. We assume by symmetry that $r\leq s$. Let $n=r-1+($the number of unit boxes below the diagonal$)$.
We define a one-to-one function from the set of Dyck paths to the set of partitions of $n$ by
$D\mapsto (\gamma_0,\gamma_1,...,\gamma_{s-1})$, where $\gamma_k=
($the number of unit boxes in the $k$-th row and below $D)$ for $k\in\{1,...,s-1\}$ and $\gamma_0=n-\sum_{k=1}^{s-1} \gamma_k$. It is straightforward to see that $\gamma_0\geq \gamma_1\geq \cdots\geq \gamma_{s-1}$. Hence we have $\mathrm{mult}\, (\alpha)\ \leq\  p_s(n)$, so it suffices to show that $n\leq 1-\frac{(\alpha |\alpha)}{2}$.

Since $\mathrm{mult}\, (\alpha)$ is invariant under the Weyl group action, we can further assume that $r\leq s \leq \frac{a}{2}r$. Then $$1-\frac{(\alpha |\alpha)}{2}=1+ars-r^2-s^2\geq 1+\frac{a^2-4}{2a}rs\geq 1+\frac{5}{6}rs\geq r-1+\frac{1}{2}rs\geq n.$$
\end{proof}

As another corollary, 
we obtain  combinatorial upper and  lower bounds for root multiplicities:

\begin{Cor} \label{cor-bounds}
We have
\[ \sum_{\substack{\mathcal D: \, \text{\rm essential} \\ \mathrm{wt} (\mathcal D)=\alpha} } c(\mathcal D) \,  \le \mathrm{mult}\, (\alpha) \, \le \# \{ \mathcal D :  \, \text{\rm essential}, \,  \mathrm{wt} (\mathcal D)=\alpha, \,  c(\mathcal D)=1 \} .\]
\end{Cor}

\begin{proof}
The inequality for the  lower bound is clear. For the upper bound, we need only to prove that
\[ \# \{ \mathcal D :  \, \text{\rm essential}, \,  \mathrm{wt} (\mathcal D)=\alpha, \,  c(\mathcal D)=-1 \} \ge \# \{ \mathcal D_0 :  \, \text{\rm essential}, \,  \mathrm{wt} (\mathcal D_0)=\alpha / 2, \,  c(\mathcal D_0)=-1 \} .\] Suppose that $\mathcal D_0$ is essential with $\mathrm{wt}(\mathcal D_0)=\alpha/2$ and $c(\mathcal D_0)=-1$. We choose $D_0 \in \mathcal D_0$. Since  $\mathrm{mult} (\alpha/2) \ge 0$,  there exists an essential $D_1$  with $\mathrm{wt}(D_1)=\alpha/2$ and $c( D_1)=1$. Then  $D:=\pi(D_0D_1)$ is essential, and we have $\mathrm{wt}(D)=\alpha$ and $c(D)=-1$. If we fix $D_1$, then the map $D_0 \mapsto D$ is injective.
\end{proof}


\begin{Exa}

Consider $\alpha=4 \alpha_1+4\alpha_2 \in \mathcal H(3)$. Then we have the following representatives of equivalence classes of essential Dyck paths and the corresponding contribution multiplicities.  

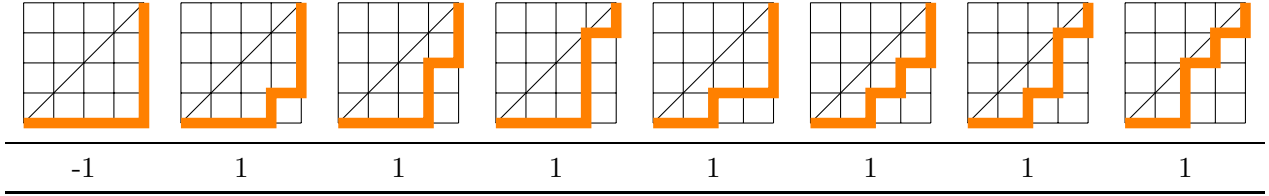
\begin{center}
\begin{tabular}{cccccccc}
\begin{tikzpicture}[scale=0.4]
 \draw (0,0) grid (4,4);
\draw (0,0)--(4,4);
\draw [orange, line width=4] (0,0) -- (4,0)--(4,4);
\end{tikzpicture}
&
\begin{tikzpicture}[scale=0.4]
 \draw (0,0) grid (4,4);\draw (0,0)--(4,4);
\draw [orange, line width=4] (0,0) -- (3,0)--(3,1)--(4,1)--(4,4);
\end{tikzpicture}
&
\begin{tikzpicture}[scale=0.4]
 \draw (0,0) grid (4,4);\draw (0,0)--(4,4);
\draw [orange, line width=4] (0,0) -- (3,0)--(3,2)--(4,2)--(4,4);
\end{tikzpicture}
&
\begin{tikzpicture}[scale=0.4]
 \draw (0,0) grid (4,4);\draw (0,0)--(4,4);
\draw [orange, line width=4] (0,0) -- (3,0)--(3,3)--(4,3)--(4,4);
\end{tikzpicture}
&
\begin{tikzpicture}[scale=0.4]
 \draw (0,0) grid (4,4);\draw (0,0)--(4,4);
\draw [orange, line width=4] (0,0) -- (2,0)--(2,1)--(4,1)--(4,4);
\end{tikzpicture}
&
\begin{tikzpicture}[scale=0.4]
 \draw (0,0) grid (4,4);\draw (0,0)--(4,4);
\draw [orange, line width=4] (0,0) -- (2,0)--(2,1)--(3,1)--(3,2)--(4,2)--(4,4);
\end{tikzpicture}
&
\begin{tikzpicture}[scale=0.4]
 \draw (0,0) grid (4,4);\draw (0,0)--(4,4);
\draw [orange, line width=4] (0,0) -- (2,0)--(2,1)--(3,1)--(3,3)--(4,3)--(4,4);
\end{tikzpicture}
&
\begin{tikzpicture}[scale=0.4]
 \draw (0,0) grid (4,4);\draw (0,0)--(4,4);
\draw [orange, line width=4] (0,0) -- (2,0)--(2,2)--(3,2)--(3,3)--(4,3)--(4,4);
\end{tikzpicture}
\\
\hline
-1 & 1 & 1 & 1 & 1 & 1 & 1 & 1\\
\hline
\end{tabular}
\end{center}
Since $A_2=4$ for $\mathcal H(3)$, the weight $\alpha/2=2\alpha_1+2\alpha_2$ does not have any path $D$ with $c(D)=-1$, and the correction term is $0$. Thus we have $\mathrm{mult}\,(\alpha)=6$.

\end{Exa} 

\begin{Exa} \label{exa-55}

Consider $\alpha= 5 \alpha_1+5\alpha_2 \in \mathcal H(3)$. Then we have the following representatives of equivalence classes of essential Dyck paths and the corresponding contribution multiplicities.  


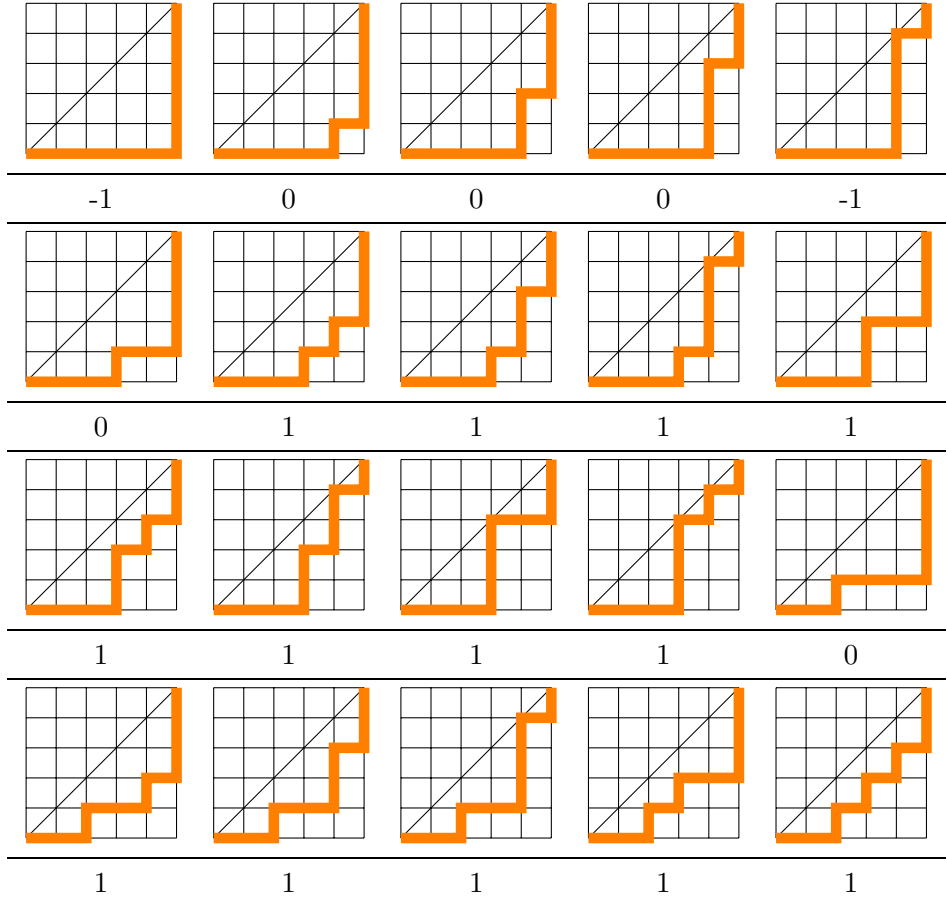
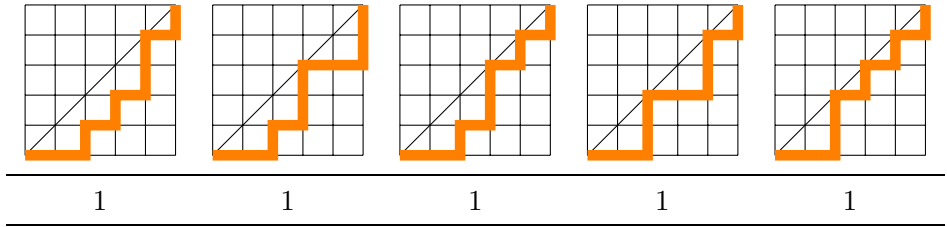
\begin{center}

\begin{tabular}{ccccc}
\begin{tikzpicture}[scale=0.4]
 \draw (0,0) grid (5,5);
\draw (0,0)--(5,5);
\draw [orange, line width=4] (0,0) -- (5,0)--(5,5);
\end{tikzpicture}
&
\begin{tikzpicture}[scale=0.4]
 \draw (0,0) grid (5,5);\draw (0,0)--(5,5);
\draw [orange, line width=4] (0,0) -- (4,0)--(4,1)--(5,1)--(5,5);
\end{tikzpicture}
&
\begin{tikzpicture}[scale=0.4]
 \draw (0,0) grid (5,5);\draw (0,0)--(5,5);
\draw [orange, line width=4] (0,0) -- (4,0)--(4,2)--(5,2)--(5,5);
\end{tikzpicture}
&
\begin{tikzpicture}[scale=0.4]
 \draw (0,0) grid (5,5);\draw (0,0)--(5,5);
\draw [orange, line width=4] (0,0) -- (4,0)--(4,3)--(5,3)--(5,5);
\end{tikzpicture}
&
\begin{tikzpicture}[scale=0.4]
 \draw (0,0) grid (5,5);\draw (0,0)--(5,5);
\draw [orange, line width=4] (0,0) -- (4,0)--(4,4)--(5,4)--(5,5);
\end{tikzpicture}
\\
\hline
-1 & 0 & 0 & 0 & -1  \\
\hline
\end{tabular}

\begin{tabular}{ccccc}
\begin{tikzpicture}[scale=0.4]
 \draw (0,0) grid (5,5);\draw (0,0)--(5,5);
\draw [orange, line width=4] (0,0) -- (3,0)--(3,1)--(5,1)--(5,5);
\end{tikzpicture}
&
\begin{tikzpicture}[scale=0.4]
 \draw (0,0) grid (5,5);
\draw (0,0)--(5,5);
\draw [orange, line width=4] (0,0) -- (3,0)--(3,1)--(4,1)--(4,2)--(5,2)--(5,5);
\end{tikzpicture}
&
\begin{tikzpicture}[scale=0.4]
 \draw (0,0) grid (5,5);\draw (0,0)--(5,5);
\draw [orange, line width=4] (0,0) -- (3,0)--(3,1)--(4,1)--(4,3)--(5,3)--(5,5);
\end{tikzpicture}
&
\begin{tikzpicture}[scale=0.4]
 \draw (0,0) grid (5,5);\draw (0,0)--(5,5);
\draw [orange, line width=4] (0,0) -- (3,0)--(3,1)--(4,1)--(4,4)--(5,4)--(5,5);
\end{tikzpicture}
&
\begin{tikzpicture}[scale=0.4]
 \draw (0,0) grid (5,5);\draw (0,0)--(5,5);
\draw [orange, line width=4] (0,0) -- (3,0)--(3,2)--(5,2)--(5,5);
\end{tikzpicture}
\\
\hline
0 & 1 & 1 & 1 & 1 \\
\hline
\end{tabular}

\begin{tabular}{cccccc}

\begin{tikzpicture}[scale=0.4]
 \draw (0,0) grid (5,5);\draw (0,0)--(5,5);
\draw [orange, line width=4] (0,0) -- (3,0)--(3,2)--(4,2)--(4,3)--(5,3)--(5,5);
\end{tikzpicture}
&
\begin{tikzpicture}[scale=0.4]
 \draw (0,0) grid (5,5);\draw (0,0)--(5,5);
\draw [orange, line width=4] (0,0) -- (3,0)--(3,2)--(4,2)--(4,4)--(5,4)--(5,5);
\end{tikzpicture}
&
\begin{tikzpicture}[scale=0.4]
 \draw (0,0) grid (5,5);
\draw (0,0)--(5,5);
\draw [orange, line width=4] (0,0) -- (3,0)--(3,3)--(5,3)--(5,5);
\end{tikzpicture}
&
\begin{tikzpicture}[scale=0.4]
 \draw (0,0) grid (5,5);\draw (0,0)--(5,5);
\draw [orange, line width=4] (0,0) -- (3,0)--(3,3)--(4,3)--(4,4)--(5,4)--(5,5);
\end{tikzpicture}
&
\begin{tikzpicture}[scale=0.4]
 \draw (0,0) grid (5,5);\draw (0,0)--(5,5);
\draw [orange, line width=4] (0,0) -- (2,0)--(2,1)--(5,1)--(5,5);
\end{tikzpicture}
\\
\hline
1 & 1 & 1 & 1 & 0  \\
\hline
\end{tabular}

\begin{tabular}{ccccc}
\begin{tikzpicture}[scale=0.4]
 \draw (0,0) grid (5,5);\draw (0,0)--(5,5);
\draw [orange, line width=4] (0,0) -- (2,0)--(2,1)--(4,1)--(4,2)--(5,2)--(5,5);
\end{tikzpicture}
&
\begin{tikzpicture}[scale=0.4]
 \draw (0,0) grid (5,5);\draw (0,0)--(5,5);
\draw [orange, line width=4] (0,0) -- (2,0)--(2,1)--(4,1)--(4,3)--(5,3)--(5,5);
\end{tikzpicture}
&
\begin{tikzpicture}[scale=0.4]
 \draw (0,0) grid (5,5);\draw (0,0)--(5,5);
\draw [orange, line width=4] (0,0) -- (2,0)--(2,1)--(4,1)--(4,4)--(5,4)--(5,5);
\end{tikzpicture}
&
\begin{tikzpicture}[scale=0.4]
 \draw (0,0) grid (5,5);
\draw (0,0)--(5,5);
\draw [orange, line width=4] (0,0) -- (2,0)--(2,1)--(3,1)--(3,2)--(5,2)--(5,5);
\end{tikzpicture}
&
\begin{tikzpicture}[scale=0.4]
 \draw (0,0) grid (5,5);\draw (0,0)--(5,5);
\draw [orange, line width=4] (0,0) -- (2,0)--(2,1)--(3,1)--(3,2)--(4,2)--(4,3)--(5,3)--(5,5);
\end{tikzpicture}
\\
\hline
1 & 1 & 1 & 1 & 1\\
\hline
\end{tabular}

\begin{tabular}{ccccc}
\begin{tikzpicture}[scale=0.4]
 \draw (0,0) grid (5,5);\draw (0,0)--(5,5);
\draw [orange, line width=4] (0,0) -- (2,0)--(2,1)--(3,1)--(3,2)--(4,2)--(4,4)--(5,4)--(5,5);
\end{tikzpicture}
&
\begin{tikzpicture}[scale=0.4]
 \draw (0,0) grid (5,5);\draw (0,0)--(5,5);
\draw [orange, line width=4] (0,0) -- (2,0)--(2,1)--(3,1)--(3,3)--(5,3)--(5,5);
\end{tikzpicture}
&
\begin{tikzpicture}[scale=0.4]
 \draw (0,0) grid (5,5);\draw (0,0)--(5,5);
\draw [orange, line width=4] (0,0) -- (2,0)--(2,1)--(3,1)--(3,3)--(4,3)--(4,4)--(5,4)--(5,5);
\end{tikzpicture}
&
\begin{tikzpicture}[scale=0.4]
 \draw (0,0) grid (5,5);\draw (0,0)--(5,5);
\draw [orange, line width=4] (0,0) -- (2,0)--(2,2)--(4,2)--(4,4)--(5,4)--(5,5);
\end{tikzpicture}
&
\begin{tikzpicture}[scale=0.4]
 \draw (0,0) grid (5,5);
\draw (0,0)--(5,5);
\draw [orange, line width=4] (0,0) -- (2,0)--(2,2)--(3,2)--(3,3)--(4,3)--(4,4)--(5,4)--(5,5);
\end{tikzpicture}
\\
\hline
1 & 1 & 1 & 1 & 1  \\
\hline
\end{tabular}
\end{center}
Since $\alpha/2$ is not an integral weight, the correction term is zero. Thus we have $\mathrm{mult}\,(\alpha)=16$.
\end{Exa}


\section{Sharper Upper Bound I}

The next goal of this paper is to obtain sharper upper bounds for root multiplicities by considering cancellation among paths with opposite contribution multiplicities. In this section, we will develop a procedure to obtain such bounds, which depends on a choice of a certain family of Dyck paths. In the next section, we will explicitly make a careful choice of such a family of Dyck paths.

We begin with a lemma, which guarantees the existence of  a family with desired properties. 

\begin{Lem} \label{lem-choice}
For each $L_{u \times v}$ of type $(-1)$, we can choose  an essential  Dyck path $M_{u\times v}$ of size $u\times v$ which contains no framed subpaths of type $(-1)$ or of type $(0)$.
\end{Lem}

\begin{proof}
Suppose that $L_{u \times v}$ is of type $(-1)$. We may assume that $u\le v$.
Since \[ \frac {A_{2n+1}-1}{A_{2n}}= \frac {aA_{2n} -A_{2n-1}}{A_{2n}}=a- \frac {A_{2n-1}}{A_{2n}} <a, \] we have $1 \le v/ u <a$. Let $E_{u \times v}$ be the Dyck path that is closest to the diagonal joining $(0,0)$ and $(u, v)$. Then $E_{u \times v}$ is given by a concatenation of subpaths of sizes $1 \times s$ with $1 \le s \le a<A_2=a+1$. Therefore, if $E_{u \times v}$ is essential, we can put $M_{u \times v} = E_{u \times v}$. If $E_{u \times v}$ is not essential, then $E_{u \times v}$ meets with the diagonal other than $(0,0)$ and $(u, v)$. Each of these intersection points is incident with a vertical edge and a horizontal edge. Immediately after the horizontal edge, $E_{u \times v}$  travels $s$ steps in the north for some $s=1, 2, \dots, a-1$. We switch the order of the vertical edge and horizontal edge, so that the resulting Dyck path, say $M_{u \times v}$, does not touch the diagonal.  Also there are $s+1<A_2$ vertical edges after the new horizontal edge.  Hence $M_{u \times v}$ is essential and does not  contain any framed subpaths of type $(-1)$ or of type $(0)$.
\end{proof}

\begin{Rmk}
The choice of $M_{u \times v}$ made in the proof of Lemma \ref{lem-choice} is not optimal for upper bounds for root multiplicities.  In Section \ref{section-sharper}, we will investigate how to make an optimal choice of $M_{u \times v}$ to obtain sharp upper bounds for root multiplicities.
\end{Rmk}

For the rest of this section, we fix $M_{u \times v}$ for each $L_{u \times v}$ of type $(-1)$ that satisfies the conditions in Lemma \ref{lem-choice}. A subpath of the form $M_{u\times v}$ is to be called of type $(1c)$. Let $\Theta(\alpha)$ be the set of equivalence classes of essential Dyck paths with weight $\alpha = r \alpha_1+s\alpha_2$.
We define   a function \begin{equation} \label{eqn-Phi} \Phi: \{\mathcal D \in \Theta(\alpha):\,c(\mathcal D)=-1 \} \longrightarrow \{\mathcal D :\,c( \mathcal D)=1\}\end{equation} as follows: Choose a representative Dyck path $D \in \mathcal D$ , and we travel from $(0,0)$ to $(r,s)$ along $D$. As soon as we encounter with a subpath of type $(-1)$ or $(1c)$, we stop traveling and define $\Phi(\mathcal D)$ as the equivalence class containing the resulting Dyck path obtained from $D$ by replacing the subpath $L_{u\times v}$ or $M_{u\times v}$ with the corresponding $M_{u\times v}$ or $L_{u \times v}$, respectively. 

\begin{Rmk} \label{rmk-phi}
Note that $\Phi(\mathcal D)$ may not be essential. The definition of $\Phi$ depends on the choice of $M_{u \times v}$ and on the choice of representatives $D$. In general, $\Phi$ is not injective.
\end{Rmk}

Now we state the main theorem of this section.

\begin{Thm} \label{thm-phi}
Let $\alpha$ be a positive root of $\mathcal H (a)$. Then we have
\[   \mathrm{mult}\, (\alpha) \le \# \{ \mathcal D :   \, \text{\rm essential}, \,  \mathrm{wt} (\mathcal D)=\alpha, \,  c(\mathcal D)=1, \, \mathcal D \text{\rm \  is not an image under }\Phi \} .\]
\end{Thm}

\begin{proof}
 
We set  \[ \begin{array}{l} N_1= \# \{ \mathcal D \in \Theta(\alpha) : \,  c(\mathcal D)=1, \, \mathcal D \notin \text{\rm Im}\, \Phi \}, \\
N_2= \# \{ \mathcal D \in \Theta(\alpha): \,  c(\mathcal D)=1, \, \mathcal D \in \text{\rm Im}\, \Phi \}, \\ N_3= \# \{ \mathcal D \in \Theta(\alpha): \,  c(\mathcal D)=-1, \, \Phi(\mathcal D) \text{\rm \  is essential} \}, \\
 N_4= \# \{ \mathcal D \in \Theta(\alpha): \,  c(\mathcal D)=-1, \, \Phi(\mathcal D) \text{\rm \  is non-essential} \}, \\ N_5= \# \{ \mathcal D \in \Theta(\alpha/2): \,  c(\mathcal D)=-1  \}. \end{array} \]
Then the identity in Theorem \ref{thm-ff} can be written as
\[ \mathrm{mult} \, (\alpha) = N_1+N_2-N_3-N_4+N_5 . \]
Note that we are proving $\mathrm{mult} \, (\alpha) \le N_1$.
Clearly, $N_2 -N_3 \le 0$, and we have only to show that $-N_4+N_5\le 0$.

Suppose that $\mathcal D_0$ is essential with $\mathrm{wt}(\mathcal D_0)=\alpha/2$ and $c(\mathcal D_0)=-1$.  By  Lemma \ref{lem-c}, $D_0$ has a framed subpath $L_{u \times v}$ of type $(-1)$. Thus $D_0 \neq \Phi(D_0)$. Set $D:=\pi (D_0\, \Phi(D_0))$. Then $D$ is essential, and we have $\mathrm{wt}(D)=\alpha$ and $c(D)=-1$. Moreover $\Phi(D)$ is non-essential by construction and the correspondence $D_0 \mapsto D$ induces an injective map from $\{ \mathcal D \in \Theta(\alpha/2): \,  c(\mathcal D)=-1  \}$ to $ \{ \mathcal D \in \Theta(\alpha): \,  c(\mathcal D)=-1, \, \Phi(\mathcal D) \text{\rm \  is non-essential} \}$. Thus we have $N_5 \le N_4$.
\end{proof}

\begin{Exa}
Consider again $\alpha= 5 \alpha_1+5\alpha_2 \in \mathcal H(3)$. We choose $M_{4 \times 4}$ and $M_{5 \times 5}$ to be
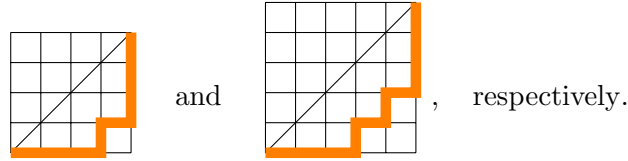
\[ \raisebox{-20 pt}{\begin{tikzpicture}[scale=0.4]
 \draw (0,0) grid (4,4);\draw (0,0)--(4,4);
\draw [orange, line width=4] (0,0) -- (3,0)--(3,1)--(4,1)--(4,4);
\end{tikzpicture}}
\quad \text{ and } \quad
\raisebox{-20 pt}{\begin{tikzpicture}[scale=0.4]
 \draw (0,0) grid (5,5);
\draw (0,0)--(5,5);
\draw [orange, line width=4] (0,0) -- (3,0)--(3,1)--(4,1)--(4,2)--(5,2)--(5,5);
\end{tikzpicture}}\ , \quad \text{respectively}.
\] Then the map $\Phi$ gives \begin{equation} \label{eqn-zp}
  \raisebox{-27 pt}
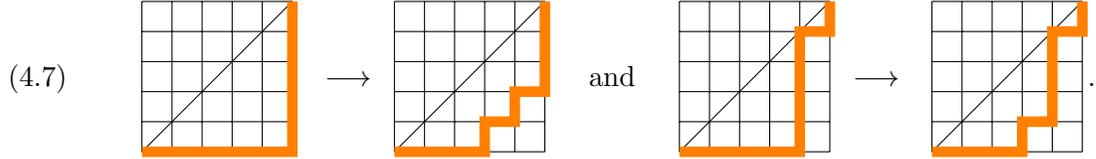
{\begin{tikzpicture}[scale=0.4]
 \draw (0,0) grid (5,5);
\draw (0,0)--(5,5);
\draw [orange, line width=4] (0,0) -- (5,0)--(5,5);
\end{tikzpicture} } \ \longrightarrow \  \raisebox{-27 pt}{\begin{tikzpicture}[scale=0.4]
 \draw (0,0) grid (5,5);
\draw (0,0)--(5,5);
\draw [orange, line width=4] (0,0) -- (3,0)--(3,1)--(4,1)--(4,2)--(5,2)--(5,5);
\end{tikzpicture}} \quad \text{ and } \quad \raisebox{-27 pt}{\begin{tikzpicture}[scale=0.4]
 \draw (0,0) grid (5,5);\draw (0,0)--(5,5);
\draw [orange, line width=4] (0,0) -- (4,0)--(4,4)--(5,4)--(5,5);
\end{tikzpicture}} 
 \ \longrightarrow \ 
\raisebox{-27 pt}{
\begin{tikzpicture}[scale=0.4]
 \draw (0,0) grid (5,5);\draw (0,0)--(5,5);
\draw [orange, line width=4] (0,0) -- (3,0)--(3,1)--(4,1)--(4,4)--(5,4)--(5,5);
\end{tikzpicture}}.
\end{equation} In the first correspondence, the whole path is $L_{5 \times 5}$ and it is simply replaced by $M_{5 \times 5}$; in the second, the subpath $L_{4 \times 4}$ is replaced by $M_{4 \times 4}$. There are $18$ essential equivalence classes of paths with contribution multiplicity $1$ as one can see from Example \ref{exa-55}. Since two of them are in the image of $\Phi$ as shown in \eqref{eqn-zp}, we actually obtain an equality $16=\mathrm{mult}\, (\alpha) = \# \{ \mathcal D \in \Theta (\alpha) :   \,   c(\mathcal D)=1, \, \mathcal D \notin \mathrm{Im} \, \Phi \}$.

\end{Exa}

\section{Sharper Upper Bound II}\label{section-sharper}

The upper bound in Theorem \ref{thm-phi} depends on the choice of Dyck paths $M_{u \times v}$ and the resulting function $\Phi$. In this section, we will make an optimal choice of $M_{u \times v}$ so that $\Phi$ may become close to an injection and consequently produce sharp upper bounds for root multiplicities. 

Recall that we have defined the sequences $\{ A_n \}$, $\{ B_n \}$ by
\[ \begin{array}{ll}
A_0=0, \quad A_1=1, \quad A_{n+2} = a A_{n+1} - A_n+1  \quad \text{ for } n \ge 0, \\
B_0=0, \quad B_1=1, \quad B_{n+2} = a B_{n+1} - B_n   \quad \text{ for } n \ge 0. \end{array} \]

\begin{Lem}\label{ratio}
We have, for $i=1, 2, \dots$, \[  A_{i+1}A_{i-1}=A_{i}^2 - A_i . \]
\end{Lem}
\begin{proof}
We use induction. If $i=1$, then the assertion is clearly true. Assume that we have $A_{i}A_{i-2}=A_{i-1}^2 - A_{i-1}$. Then we obtain
$$\aligned
A_{i+1}A_{i-1} &= (aA_i-A_{i-1}+1)A_{i-1} = a A_iA_{i-1}-A_{i-1}^2 +A_{i-1}\\
&=aA_i A_{i-1} -A_iA_{i-2}= A_i (a A_{i-1}-A_{i-2})=A_i(A_i-1)=A_i^2-A_{i}.\endaligned$$
\end{proof}

For any positive integers $u,v$, denote by $E_{u\times v}$ the Dyck path of size $u\times v$ that is closest to the diagonal joining $(0,0)$ and $(u,v)$.
For any integer $n\geq 2$, we define
{\small $$\aligned
&M_{A_{2n}\times (A_{2n+1}-i)}\\
&=\left\{
\begin{array}{ll}
L_{A_{2n-1}\times (A_{2n}-2)}E_{(A_{2n}-2A_{2n-1})\times(A_{2n+1}-2A_{2n}+2)}L_{A_{2n-1}\times (A_{2n}-1)}, & \text{ for }i=1;\\
L_{A_{2n-1}\times (A_{2n}-3)}E_{(A_{2n}-2A_{2n-1})\times(A_{2n+1}-2A_{2n}+4-i)}L_{A_{2n-1}\times (A_{2n}-1)}, & \text{ for }2\leq i\leq a+2;\\
L_{(A_{2n-1}+1)\times j_1}E_{(A_{2n}-2A_{2n-1}-1)\times(A_{2n+1}-A_{2n}+1-j_1 -i )}L_{A_{2n-1}\times (A_{2n}-1)}, & \text{ for }a+3\leq i\leq \frac{(a+3)A_{2n}}{A_{2n-1}};\\
L_{(A_{2n-1}+1)\times j_1}E_{(A_{2n}-2A_{2n-1}-1)\times(A_{2n+1}-j_1-j_2 -i )}L_{A_{2n-1}\times j_2}, &\text{ for } \frac{(a+3)A_{2n}}{A_{2n-1}}<i \le A_{2n+1}-A_{2n},\\
\end{array}
\right.
\endaligned$$}
where $j_1$ is the integer satisfying
$$
\frac{j_1}{A_{2n-1}+1} \leq \frac{A_{2n+1}-i}{A_{2n}} < \frac{j_1+1}{A_{2n-1}+1}
$$
and $j_2$ is the integer satisfying
$$
\frac{j_2-1}{A_{2n-1}+1} < \frac{A_{2n+1}-i}{A_{2n}} \leq \frac{j_2}{A_{2n-1}+1}.
$$

For $1\leq i\leq A_{2n+1}-A_{2n}-\frac{A_{2n}}{A_{2n-1}+1}$, we define $M_{(A_{2n+1}-i)\times A_{2n}}$ to be the transpose of $M_{A_{2n}\times (A_{2n+1}-i)}$.
For $A_{2n+1}-A_{2n}-\frac{A_{2n}}{A_{2n-1}+1} <i<A_{2n+1}-A_{2n}$, define $M_{(A_{2n+1}-i)\times A_{2n}}$ by
$$
L_{j_2\times A_{2n-1}} E_{(A_{2n+1}-j_1-j_2 -i)\times (A_{2n}-2A_{2n-1}-2)} L_{j_1\times (A_{2n-1}+2)},
$$
where $j_2$ is the integer satisfying
$$
\frac{A_{2n-1}}{j_2} \leq \frac{A_{2n}}{A_{2n+1}-i} < \frac{A_{2n-1}}{j_2-1}
$$
and $j_1$ is the integer satisfying
$$
\frac{A_{2n-1}+2}{j_1+1} < \frac{A_{2n}}{A_{2n+1}-i} \leq \frac{A_{2n-1}+2}{j_1}.
$$

%

\begin{Lem}\label{still_Dyck}
The paths $M_{A_{2n}\times (A_{2n+1}-i)}$ and $M_{(A_{2n+1}-i)\times A_{2n}}$ are Dyck paths for each $n \ge 2$ and $1 \le i \le A_{2n+1}-A_{2n}$.
\end{Lem}
\begin{proof}
Since the other case is similar, we only consider $M_{A_{2n}\times (A_{2n+1}-i)}$.
First we need to show that $\frac{A_{2n+1}-i}{A_{2n}}  \leq     \frac{A_{2n}-1}{A_{2n-1}}$ which is equivalent to
$$
A_{2n+1}A_{2n-1}-iA_{2n-1} \leq A_{2n}^2-A_{2n}.
$$
By Lemma~\ref{ratio}, this becomes $iA_{2n-1}\geq 0$, which is obvious.

Next, we consider $\frac{A_{2n+1}-1}{A_{2n}}  \geq     \frac{A_{2n}-2}{A_{2n-1}}$
which is equivalent to 
$$
A_{2n+1}A_{2n-1}-A_{2n-1} \ge A_{2n}^2-2A_{2n}.
$$
Again by Lemma~\ref{ratio}, this becomes $A_{2n-1} \le A_{2n}$, which is clearly true. The remaining cases can be checked in a similar way.
\end{proof}

For $n\geq 2$ and $k\in\{1,2, \dots , A_{2n+1}-A_{2n}-1\}$, we define
{\small $$\aligned
&M_{(A_{2n}+k)\times (A_{2n+1}-i)}\\
&=\left\{
\begin{array}{ll}
L_{(A_{2n-1}+1)\times j_1}E_{(A_{2n}-2A_{2n-1}+k-1)\times(A_{2n+1}-A_{2n}+1-j_1 -i)}L_{A_{2n-1}\times (A_{2n}-1)}, & \text{ for }1\leq i\leq p;\\
L_{(A_{2n-1}+1)\times j_1}E_{(A_{2n}-2A_{2n-1}+k-1)\times(A_{2n+1}-j_1-j_2 -i)}L_{A_{2n-1}\times j_2}, &\text{ for } p<i\leq A_{2n+1}-A_{2n}-k,\\
\end{array}
\right.
\endaligned$$}
where $j_1$ is the integer satisfying
$$
\frac{j_1}{A_{2n-1}+1} \leq \frac{A_{2n+1}-i}{A_{2n}+k} < \frac{j_1+1}{A_{2n-1}+1},
$$
$j_2$ is the integer satisfying
$$
\frac{j_2-1}{A_{2n-1}+1} < \frac{A_{2n+1}-i}{A_{2n}+k} \leq \frac{j_2}{A_{2n-1}+1},
$$
and $p$ is the integer satisfying $$ 
\frac{(A_{2n}-3)-1}{A_{2n-1}+1} < \frac{A_{2n+1}-p}{A_{2n}+k} \leq \frac{(A_{2n}-3)}{A_{2n-1}+1}.
$$

For $1\leq i\leq A_{2n+1}-A_{2n}-k-\frac{A_{2n}+k}{A_{2n-1}+1}$, we define $M_{(A_{2n+1}-i)\times (A_{2n}+k)}$ to be the transpose of $M_{(A_{2n}+k)\times (A_{2n+1}-i)}$.
For $A_{2n+1}-A_{2n}-k-\frac{A_{2n}+k}{A_{2n-1}+1} <i<A_{2n+1}-A_{2n}-k$, define $M_{(A_{2n+1}-i)\times (A_{2n}+k)}$ by
$$
L_{j_2\times A_{2n-1}} E_{(A_{2n+1}-j_1-j_2 -i)\times (A_{2n}-2A_{2n-1} +k -2)} L_{j_1\times (A_{2n-1}+2)},
$$
where $j_2$ is the integer satisfying
$$
\frac{A_{2n-1}}{j_2} \leq \frac{A_{2n}+k}{A_{2n+1}-i} < \frac{A_{2n-1}}{j_2-1}
$$
and $j_1$ is the integer satisfying
$$
\frac{A_{2n-1}+2}{j_1+1} < \frac{A_{2n}+k}{A_{2n+1}-i} \leq \frac{A_{2n-1}+2}{j_1}.
$$

\begin{Lem}
The paths $M_{(A_{2n}+k)\times (A_{2n+1}-i)}$ and $M_{(A_{2n+1}-i)\times (A_{2n}+k)}$ are Dyck paths for each $n\ge2$, $1 \le k \le A_{2n+1}-A_{2n}-1$ and $1 \le i \le A_{2n+1}-A_{2n}-k$.
\end{Lem}

\begin{proof}
The proof is even simpler than the proof of Lemma~\ref{still_Dyck}, so we will omit it.
\end{proof}

We recall the definition of mutation of Dyck paths, which is  developed by Lee-Schiffler \cite{LS}, Rupel \cite{R} and Lee-Li-Zelevinsky \cite{LLZ}.
	\begin{Def}
	Consider the bijective function $\phi:\{1^a2,1^{a-1}2,\cdots,12,2\}\longrightarrow\{1,12,\cdots,12^{a-1},12^a\}$ defined by 
$$\phi(1^a2)=1, \quad \phi(1^{a-1}2)=12, \quad  \dots \quad \phi(12)=12^{a-1}, \quad \phi(2)=12^{a}.$$	
	Suppose that a finite sequence $S$ is obtained by concatenating (copies of) $1^a2,1^{a-1}2,\cdots,12,2$. Let $\phi(S)$ be the sequence obtained from $S$ by replacing each subsequence $1^a2$ (resp. $1^{a-1}2$, $\cdots$, $2$) with $\phi(1^a2)$ (resp. $\phi(1^{a-1}2)$, $\cdots$, $\phi(2)$). We call $\phi(S)$	the {\em mutation} of $S$.  \end{Def}
 	
\begin{Lem}
Let $u$ and $v$ be positive integers with $0\leq av-u\leq v\leq u$.   Then $\phi(E_{u\times v})=E_{v\times (av-u)}$.
\end{Lem}
\begin{proof}
See  \cite[p.68]{LLZ}.
\end{proof}

Now, for each $(u,v)$ with $A_{2}\leq \min (u,v), \ \max(u,v) < A_{3}$, we choose  an essential  Dyck path $M_{u\times v}$ of size $u\times v$ such that $M_{u\times v}\neq E_{u\times v}$ except for $(u,v)=(A_2,A_3-1)$ and $(A_3-1,A_2)$ and which contains no framed subpaths of type $(-1)$ or of type $(0)$.
A subpath of the form $M_{u\times v}$ with $A_{2n}\leq \min (u,v), \ \max(u,v) < A_{2n+1}$ ($n \ge 1$) is said to be of type $(1s)$. More specifically, a subpath of the form $M_{u\times v}$ with $A_{2}\leq \min (u,v), \ \max(u,v) < A_{3}$ is said to be of type $(1s1)$. If $n\geq 2$, then a subpath of the form $M_{u\times v}$ with $A_{2n}\leq \min (u,v), \ \max(u,v) < A_{2n+1}$ is said to be of type $(1s2)$. Here we do not further specify the paths $M_{u \times v}$ of type $(1s1)$ since the conditions given above are enough to obtain the main result (Proposition \ref{prop-last}), whereas the paths $M_{u\times v}$ of type $(1s2)$ have been deliberately chosen in this section.

The following proposition shows that our choice of $M_{u \times v}$ made above for type $(1s2)$ is optimal in the sense that the resulting map $\Phi$ in \eqref{eqn-Phi} is as close to an injection as possible. Before stating the proposition, we define one more terminology: For a Dyck path $D$, a pair of two subpaths of $D$ is said to be \emph{disjoint} if the two subpaths share no edges.

\begin{Prop} \label{prop-last}
Suppose that $M_{u\times v}\neq E_{u\times v}$ for $A_{2}\leq \min (u,v), \ \max(u,v) < A_{3}$ except for $(u,v)=(A_2,A_3-1)$ and $(A_3-1,A_2)$. Consider a Dyck path $D$. Then
each subpath of type $(1s2)$ of $D$  is disjoint from all other subpaths of type $(1s)$.
\end{Prop}

\begin{proof}
First we show that any pair of two distinct subpaths of type $(1s2)$ is disjoint. Suppose that two distinct subpaths of type $(1s2)$ are not disjoint. Let  one of the two subpaths be $L_{u_1\times v_1}E_{u_2\times v_2}L_{u_3\times v_3}$ and the other $L_{u_1'\times v_1'}E_{u_2'\times v_2'}L_{u_3'\times v_3'}$. Without loss of generality, assume that the first subpath starts before the second one does.      Since $u_1,v_1,u_3,v_3,u_1',v_1',u_3',v_3'>2$,      the only possibility is that  $u_3=u_1'$ and $v_3=v_1'$. We will check that this never happens. By symmetry we further assume  that $u_3\leq v_3$.

Then, by definition of $j_2$, we get 
$$\aligned
j_2-1 
&< \frac{A_{2n-1}+1}{A_{2n}}\left(A_{2n+1}-\frac{(a+3)A_{2n}}{A_{2n-1}} \right) \\ 
&= \frac{A_{2n-1}A_{2n+1}}{A_{2n}} + \frac{A_{2n+1}}{A_{2n}} -(a+3) -\frac{a+3}{A_{2n-1}}\\
&\overset{Lemma~\ref{ratio}}{=} A_{2n}-1 + \frac{A_{2n+1}}{A_{2n}} -(a+3) -\frac{a+3}{A_{2n-1}}\\
&<A_{2n}-1 + a-(a+3) = A_{2n}-4, \\
\endaligned$$  
which implies that $j_2\leq A_{2n}-4$. So if $u_3=u_1'=A_{2n-1}$ then $v_3\neq v_1'$. For other cases, it is easy to see that if $u_3=u_1'$ then $v_3\neq v_1'$. 

Next we show that each subpath of type $(1s2)$  is disjoint from any subpath of type $(1s1)$.  Let $W$ be a subpath of type $(1s2)$ and $V$ a subpath of type $(1s1)$.    Note that $W$ is of the form $L_{u_1\times v_1}E_{u_2\times v_2}L_{u_3\times v_3}$ with $\min(u_1,v_1,u_3,v_3)\geq A_3$, and that $V$ is of the form     $M_{u\times v}$ with $A_{2}\leq \min (u,v), \ \max(u,v) < A_{3}$. So if $W$ and $V$ are not disjoint, then $M_{u\times v}$ should be a subpath of  $E_{u_2\times v_2}$. This happens only when $M_{u\times v}= E_{u\times v}$, in other words, $M_{u\times v}$ is closest to the diagonal. Hence $(u,v)=(A_2,A_3-1)$ or $(A_3-1,A_2)$.  Without loss of generality, let $(u,v)=(A_3-1,A_2)$.  However $(\phi\circ\phi)(E_{u\times v})$ is not defined, because $aA_2-(A_3-1)=A_1=1$ and $aA_1-A_2<0$. On the other hand, it is straightforward to check that $(\phi\circ\phi)(E_{u_2\times v_2})$ is well defined, which implies that  $E_{u\times v}$ cannot be a subpath of  $E_{u_2\times v_2}$. 
\end{proof}

\end{document}